\documentclass[12pt]{amsart}
\usepackage{amssymb,amsmath,amscd}
\usepackage{graphicx}
\usepackage{color}
\usepackage{subfigure}
\graphicspath{{FIGURES/}}

\newtheorem{theorem}{Theorem}[section]
\newtheorem{lemma}[theorem]{Lemma}

\newtheorem{proposition}[theorem]{Proposition}

\newtheorem{algorithm}[theorem]{Algorithm}

\newcommand*{\doublebarsim}{%
\mathrel{\vcenter{\offinterlineskip\hbox{$=$}
\vskip-.10ex\hbox{$\sim$}}}}
\newcommand{\ott}[1]{\underset{\widetilde{\widetilde{}}}{#1}}
\newcommand{\ot}[1]{\underset{\widetilde{}}{#1}}

\begin{document}

\title[An Auxiliary Space Preconditioner for linear elasticity based on GFEM]
{Auxiliary space preconditioners for linear elasticity based on generalized finite element methods}

\author{James Brannick}
\address{Department of Mathematics, The Pennsylvania State
University, University Park, PA 16802, USA.}
\email{brannick@math.psu.edu}

\author{Durkbin Cho}
\address{Department of Mathematics, The Pennsylvania State
University, University Park, PA 16802, USA.}
\email{cho@math.psu.edu}

\subjclass[2000]{65F10, 65N30, 65N55}
\keywords{linear elasticity, preconditioning, auxiliary space}
\date{\today}

\maketitle

\begin{abstract}  
We construct and analyze a preconditioner of the linear elastiity
system discretized by conforming linear finite elements in the
framework of the auxiliary space method. The auxiliary space
preconditioner is based on discretization of a scalar elliptic
equation with Generalized Finite Element Method (GFEM).
 \end{abstract}

\section{Introduction}
The discretizations and fast solvers for the linear elasticity systems
have been extensive subject of research for the past 20 years. A
number of methods have been introduced
in~\cite{Arnold_Winther,Brenner_1994,Falk_1991,Franca_Stenberg,Lee_1998_2,Schoberl},
and in these works one can find stable discretizations, a priori error
estimates, as well as construction of fast and robust solvers for
different range of material parameters.

In this paper we consider the variational problem corresponding to the
lowest order finite element method discretization of linear elasticity
system in displacement formulation.  We assume that the material
parameters in the linear elasiticity system are well behaved, namely,
the Poisson ratio is away from $1/2$. We will not discuss the
robustness of our method with respect to these parameters, because our
goal here is to introduce and prove results on the relationship
between GFEM discretizations of scalar equations and the linear
elastitcity system, as well as to employ such relations in the
construction of auxiliary space~\cite{Nepomnyaschikh_1992,JXu_aux}
preconditioner. 

The preconditioner that we construct employs as auxiliary space the
piecewise quadratic conforming finite elements, and the corresponding
auxiliary bilinear form on the auxiliary space corresponds to the
discretized scalar Laplace's equation. To relate the discretization of
the scalar Laplace problem to linear elasticity, we also use an
``intermediate'' GFEM space containing piece-wise quadratic functions.
The key steps in the analysis of the preconditioner rely on the spectral
equivalence results proved in~\cite{Cho_Zikatanov2006} and the fact
that the kernel of the quadratic GFEM stiffness matrix is isomorphic
to the space of rigid modes (see~\cite{Cho_Zikatanov2006}).

The remainder of the paper is organized as follows. In section 2, we
present the linear elasticity problem of interest and its variational
formulation. In section 3, we show an auxiliary spectral equivalence
relation that plays a key role in the analysis. We introduce the
preconditioner, and prove uniform spectral bounds in section
4. Numerical results are presented in section 5 are shown to validate
the theoretical results.

\section{Linear Elasticity}

Let $\Omega\subset\mathbb{R}^2$ be a bounded convex polygonal domain with
boundary $\Gamma=\partial \Omega$. We consider
the linear elasticity
problem with pure traction boundary conditions:\\
\begin{equation}\label{eqtn:linear_elasticity}
\begin{split}
-\ot{\rm{div}}\big(2\mu\varepsilon(\ot{u})
+\lambda\mathrm{tr}(\varepsilon(\ot{u}))\ott{\delta}\big)&=\ot{f}\quad \mbox{in}\ \Omega,\\
\big(2\mu\varepsilon(\ot{u})
+\lambda\mathrm{tr}(\varepsilon(\ot{u}))\ott{\delta}\big)\ot{n}
&=\ot{g}\quad \mbox{on}\ \Gamma,
\end{split}
\end{equation}
where $\ot{f}$ is an external force, $\ot{n}$ is the
outward unit normal on the boundary,
$\varepsilon_{ij}=\frac12(\partial_j u_i +
\partial_i u_j)$ is the strain tensor, and $\ott{\delta}$ is
a matrix whose elements consist of the
Kronecker delta symbol. Further, let
{\rm{tr}} denote the trace of a matrix and define the Lam\'{e} coefficients $\mu$ and $\lambda$
in terms of the Young modulus $E$ and the Poisson ratio $\nu$ as follows:
\[  \lambda = \frac{E\nu}{(1+\nu)(1-2\nu)}\quad\mbox{and}\quad
\mu=\frac{E}{2(1+\nu)}\]
We mention that due to the corners of the boundary of the polygonal domain $\Omega$, care
must be taken when considering the boundary conditions of~\eqref{eqtn:linear_elasticity} (See~\cite{Brenner_1994,Brenner_Sung_Elasticity} for details).

Let $S_i$, $1\le i\le n$, be the vertices of $\Gamma$, $\Gamma_i$, $1\le i\le n$,the open line segments joining $S_i$ to $S_{i+1}$, and $\ot{n} \hskip 0mm_i$ the
unit outer normal along $\Gamma_i$. Let $p\in H^{1/2}(\Gamma_i)$ and $q\in H^{1/2}(\Gamma_{i+1})$.
Then, $p\equiv q$ at $S_{i+1}$ if
\[\int_0^\delta |q(s)-p(-s)|^2\frac{ds}{s}<\infty,  \]
where $s$ is the oriented arc length measured from $S_{i+1}$, and $\delta$ is a positive number less than
$\min\{|\Gamma_i|: 1\le i\le n\}$. Then, equation~\eqref{eqtn:linear_elasticity} can be written more precisely as
\begin{equation}\label{eqtn:linear_elasticity_precise}
\begin{split}
-\ot{\rm{div}}\big(2\mu\varepsilon(\ot{u})
+\lambda\mathrm{tr}(\varepsilon(\ot{u}))\ott{\delta}\big)&=\ot{f}\quad \mbox{in}\ \Omega\\
\big(2\mu\varepsilon(\ot{u})
+\lambda\mathrm{tr}(\varepsilon(\ot{u}))\ott{\delta}\big)\ot{n}\hskip0mm_i\big|_{\Gamma_i}
&=\ot{g}\hskip0mm_i \quad 1\le i \le n,
\end{split}
\end{equation}
where $\ot{f}\in \ot{L}\hskip0mm_2(\Omega)$ and $\ot{g}\hskip0mm_i\in \ot{H}\hskip0mm^{1/2}(\Gamma_i)$
satisfy
\[\ot{g}\hskip0mm_i\cdot \ot{n}\hskip0mm_{i+1}\equiv
\ot{g}\hskip0mm_{i+1}\cdot\ot{n}\hskip0mm_i\quad \mbox{at}\quad
 S_{i+1}\quad \mbox{for}\quad 1\le i \le n.\]

Denote by $\rm{RM}$ and $\mathrm{RM}^\perp$ the space of rigid modes and its
orthogonal complement with respect to $\ot{L} \hskip 0mm_2(\Omega)$,
respectively and by $\ot{H}\hskip0mm_\perp^1(\Omega)
:=\ot{H}^1(\Omega)\cap
\mathrm{RM}^\perp$ the intersection of $\ot{H}^(\Omega)$
with the $\ot{L}\hskip0mm_2$ orthogonal complement
$\mathrm{RM}^\perp$ of $\mathrm{RM}$. From integration by parts,
we obtain the weak formulation of~\eqref{eqtn:linear_elasticity_precise} as follows:\\
\indent Find $\ot{u}\in
\ot{H}\hskip0mm_\perp^1(\Omega)$
such that
\begin{equation}\label{variational form}
a_{LE}(\ot{u},\ot{v})=\int_\Omega\ot{f}\cdot\ot{v}~dx+\sum_{i=1}^n \int_{\Gamma_i} \ot{g} \hskip 0mm_i\cdot
\ot{v}|_{\Gamma_i}~ds
\end{equation}
for all $\ot{v}\in \ot{H}\hskip0mm_\perp^1(\Omega)$,
where
\[a_{LE}(\ot{u},\ot{v}):=\int_\Omega
\left(2\mu\langle\varepsilon(\ot{u})
,\varepsilon(\ot{v})\rangle_F+\lambda(\nabla\cdot\ot{u})
(\nabla\cdot\ot{v})\right)dx.\]
Here and below,  $\langle\cdot,\cdot\rangle_F$ and
$\langle\cdot,\cdot\rangle$ are the Frobenius inner product for
matrices and the standard Euclidean inner product for vectors in
$\mathbb{R}^d$, respectively. The corresponding norms are denoted with
$|\cdot|_F$ and $|\cdot|$, respectively. A sufficient condition for existence of a solution to~\eqref{variational form} is that the
following compatibility condition is satisfied:
\[ \int_\Omega \ot{f}\cdot\ot{v}~dx+
\sum_{i=1}^n\int_{\Gamma_i}
\ot{g} \hskip 0mm_i\cdot\ot{v}\big|_{\Gamma_i}~ds=0,\quad
\forall \ot{v}\in \mathrm{RM}.\]

Following~\cite{JXu_SIAM_Review},
we write $x_1\lesssim y_1, x_2\gtrsim
y_2$ and $x_3\doublebarsim y_3$ whenever there exist constants $C_1,c_2,c_3$
and $C_4$, independent of the mesh size $h$,
such that
\[x_1\le C_1 y_1,\quad x_2\ge c_2 y_2,\quad\mbox{and}\quad
c_3 x_3\le y_3 \le C_4x_3.\]
Moreover, we write $\lesssim_\lambda, \gtrsim_\lambda$ and
$\doublebarsim_\lambda$ when $\lesssim, \gtrsim$ and
$\doublebarsim$ are dependent on the the Lam\'{e} coefficient $\lambda$.

\section{A multilevel preconditioning for the GFEM problems}

In this section, we present our GFEM-based preconditioner,
using results from~\cite{Cho_Zikatanov2006} to motivate our scheme.  Let
$\Omega\subset\mathbb{R}^d$ be a polygonal domain, with $d=2,3$ and  $V_1$ be the associated
piecewise linear finite element space on a quasi-uniform triangulation
$\mathcal{T}_h$ of $\Omega$. Then the GFEM space $V^{GFEM}$ is
defined as
\[V^{GFEM}:=V_1+\sum_{k=1}^d\sum_{i=1}^d \mbox{span}\{\psi_{k,i}\phi_i\},\]
where $V_1$ corresponds to continuous piecewise linear elements
$\{\phi_i\}_{i=1}^n$ and $\psi_{k,i}=\frac{x_k-x_{k,i}}{h}$ with $x_k$ the
$k$-th component of $x$ and $x_{k,i}$ the $k$-th value associated with the
$i$-th nodal point. Consider the bilinear form
\begin{equation}\label{weak elliptic}
a(u,v)=\int_\Omega \nabla u\cdot \nabla v,\quad\forall u, v\in V^{GFEM}.
\end{equation}
It is shown in~\cite{Cho_Zikatanov2006} that the kernel of the stiffness matrix $A$ corresponding to~\eqref{weak elliptic} is characterized by the space $\mathrm{RM}$ of rigid modes. To obtain an efficient preconditioner for
$a(\cdot,\cdot)$ on $V^{GFEM}$, we thus choose as auxiliary space
$W=V_1\times V_1^d$ defined by
\[\mathbf{u}\in W,\quad \mathbf{u}=(u,\ot{u}),\quad \ot{u}=(u_k)\quad \mbox{where}
\quad u=\sum_{i=1}^n \alpha_i \phi_i,\ u_k=\sum_{i=1}^n
\alpha_{k,i}\phi_i,\] equipped with
$a_W(\mathbf{u},\mathbf{v})=a(u,v)+\int_\Omega\langle
\varepsilon(\ot{u}),\varepsilon(\ot{v})\rangle_F~dx$. Define $\Pi:W\rightarrow V^{GFEM}$ to be:
\[ \mathbf{u}\rightarrow \Pi\mathbf{u}=\sum_{i=1}^n \alpha_i\phi_i +\sum_{k=1}^d\sum_{i=1}^n \alpha_{k,i}\psi_{k,i}\phi_i, \]
and take $A_G:V^{GFEM}\rightarrow (V^{GFEM})'$ and $A_W:W\rightarrow
W'$ to be the isomorphisms associated with $a(\cdot,\cdot)$ and
$a_W(\cdot,\cdot)$, respectively.  Here, $'$ and $*$ denote the dual
spaces and adjoint operators.

We consider the case when the linear part of the GFEM elements are zero: $W = \{0\}\times V^d_1$, so that
$\Pi:W\rightarrow
V^{GFFEM}$, or more precisely,
\[ \mathbf{u}=(0,\ot{u})\longrightarrow \Pi\mathbf{u}=0+\sum_{k=1}^d\sum_{i=1}^n \alpha_{k,i}\psi_{k,i}\phi_i, \]
or
\[ \ot{u}\in W\longrightarrow \Pi\ot{u}=\sum_{k=1}^d\sum_{i=1}^n \alpha_{k,i}\psi_{k,i}\phi_i\in V^{GFEM}.
\]
We remark that $\rm{ker}(A_G)$ is isomorphic to $\rm{ker}(A_W)$ where ker denotes the kernel of the operator. The following spectral equivalence, proved in the next subsection, then holds:
\begin{equation}\label{CZ_specEq}
a_W(\ot{u},\ot{u}) \doublebarsim
a(\Pi\ot{u},\Pi\ot{u}),\quad \forall \ot{u}\in W.
\end{equation}
Now, by the auxiliary space lemma~\cite{Nepomnyaschikh_1992,JXu_aux},
$\Pi B \Pi^*$ is a uniform
preconditioner for $A_G$, namely, $\kappa(\Pi B
\Pi^*A_G)\lesssim 1$ where $B$ is a BPX preconditioner for $A_W$ (see~\cite{Cho_Zikatanov2006}
for a detailed proof).
$$\begin{CD}
W @>A_W>> W'\\
@V\Pi VV  @AA\Pi^* A\\
V^{GFEM} @>>A_G> (V^{GFEM})'
\end{CD}$$
Hereafter, we use $\Pi_G$ for the GFEM interpolant associated with the GFEM space, $V^G$.

\subsection{Spectral Equivalence}  Consider the following bilinear form defined for piecewise linear,
continuous vector fields on a triangulation of $\Omega$ with
simplexes $\mathcal{T}_h$
\[
a_W(\ot{u},\ot{v}) : = \int_\Omega
\langle \varepsilon(\ot{u}),\varepsilon(\ot{v})\rangle_F.
\]

For a given piecewise linear continuous vector field, $\ot{u}$, denote the corresponding
element in the GFEM space by $u_G$. Consider
the isomorphism, $\Pi$, between the GFEM space (with $0$ linear part)
and the piece-wise linear continuous vector fields in $\mathbb{R}^d$: $$u_G=\Pi\ot{u}=\sum_{k=1}^d\sum_{i=1}^n
 \alpha_{k,i}\psi_{k,i}\phi_i.$$  By direct computation
\[
u_G(\mathbf{x}) = \frac{\langle \mathbf{x}, \ot{u}\rangle -\langle \mathbf{x}, \ot{u}\rangle_I}{h},
\]
where $w_I$ denotes the continuous linear interpolant of a function $w$.
We now arrive at the following lemma.

\begin{lemma}
The following relations hold for any fixed simplex $T$ of the
triangulation and any $\mathbf{x}\in T$:
\begin{equation}\label{eq:ug1}
u_G(\mathbf{x}) =
\frac1h(\langle \varepsilon(\ot{u}) \mathbf{x},\mathbf{x}\rangle
-\langle \varepsilon(\ot{u}) \mathbf{x},\mathbf{x}\rangle_I).
\end{equation}
\end{lemma}
\begin{proof}
The proof follows by taking a Taylor expansion of $\ot{u}$ and using the fact that
$\ot{u}$ is a linear vector field on $T$ and hence $\nabla
\ot{u}$ is a constant matrix on $T$:
\[
\langle\ot{u}(\mathbf{x}),\mathbf{x}\rangle =
\langle\ot{u}(\mathbf{x}_0),\mathbf{x}\rangle +
\langle[\nabla\ot{u}]\mathbf{x},\mathbf{x}\rangle  -
\langle[\nabla\ot{u}]\mathbf{x}_0,\mathbf{x}\rangle.
\]
Taking the linear interpolant on both sides and subtracting leads to
\[
u_G(\mathbf{x}) = \frac1h
(\langle [\nabla\ot{u}] \mathbf{x},\mathbf{x}\rangle
-\langle [\nabla\ot{u}] \mathbf{x},\mathbf{x}\rangle_I).
\]
The proof is concluded by observing that for any matrix $Z\in
\mathbb{R}^{d\times d}$ we have
\[
\langle Z \mathbf{x},\mathbf{x}\rangle
=\frac12\langle (Z +Z^T)\mathbf{x},\mathbf{x}\rangle,
\]
where $Z^T$ is the transpose of $Z$,
\end{proof}

Writing out $1=\sum_{i=1}^{d+1}\lambda_i$ and
$\mathbf{x} = \sum_{i=1}^{d+1}\lambda_i \mathbf{x}_i$, where $\mathbf{x}_i$ are the vertices of $T$, and $\lambda_i(\mathbf{x})$ are the barycentric coordinate functions, we obtain that
\begin{equation}\label{eq definition of ug}
u_G=\sum_{E\subset T}
\langle\varepsilon(\ot{u})\frac{(x_i-x_j)}{h},x_i-x_j\rangle
\varphi_E(\mathbf{x})
\end{equation}
where $\varphi_E=\lambda_i\lambda_j$ on the edge
$\overline{x_ix_j}$. Differentiating and taking the $L_2$ norm we have
\begin{equation}
\|\nabla
u_G\|_{0,T}^2=\sum_{E^\prime \subset T}\sum_{E\subset T}\Big(\varepsilon(\ot{u})\frac{x_E}{h},x_E
\Big)\Big(\int_T \nabla \varphi_E\cdot\nabla\varphi_{E'}\Big)
\Big(\varepsilon(\ot{u})\frac{x_{E'}}{h},x_{E'}\Big),
\end{equation}
where $x_E=x_i-x_j$.
The proof of the spectral equivalence lemma uses the following
technical result.
\begin{proposition}\label{proposition matrix}
Let $Z\in \mathbb{R}_{\text{sym}}^{d\times d}$ be a symmetric matrix
and $\widetilde{T}$ be a nondegenerate simplex in $\mathbb{R}^d$ whose
edges of size (independent on $h$). Then
\[
|Z|^2_F \eqsim \sum_{E\subset \widetilde{T}} (\langle Z y_E,
y_E\rangle)^2,
\]
where the constants of equivalence depend on the spatial dimension $d$, and $y_E$ is a vector parallel to the edge $E$.
\end{proposition}
\begin{proof} We first prove that the following expression is an inner product on
$\mathbb{R}_{\text{sym}}^{d\times d}$:
\[
\langle Y, Z\rangle_*:=\sum_{E\subset \widetilde T}
(\langle Z y_E, y_E\rangle
\langle Y y_E, y_E\rangle).
\]
First, we aim to establish that
\[
\langle Z, Z \rangle_* \ge 0,\qquad \mbox{and}\qquad
\langle Z, Z \rangle_* =0 \quad\mbox{iff}\quad Z\equiv 0.
\]
We now provide a detailed proof for
$d=3$ case.  A similar proof applies
to the $d = 2$ case.

Assume that $\langle Z , Z\rangle_* = 0$. This implies that $\langle
Zy_E,y_E \rangle = 0$ for each $E\subset \widetilde T$.  Since the
size of $\widetilde T$ is independent of $h$, we can map it onto the
canonical simplex $\hat{T}$ in $\mathbb{R}^d$ (the convex hull of the
canonical coordinate vectors in $\mathbb{R}^d$). Clearly such mapping
is affine and independent of $h$ as well. Thus, we can limit our proof to
$\hat{T}$. Note that three of
the edges of the canonical simplex are parallel to the coordinate
vectors $\{e_k\}_{k=1}^3$, and there are three more, which are
parallel to the vectors $e_{kl}=(e_k + e_l)$, $1\le k < l\le 3$. Let
$Z=
\begin{pmatrix}
a_1 & a_{12} & a_{13} \\
a_{12} & a_2 & a_{23} \\
a_{13} & a_{23} & a_3
\end{pmatrix}$.
From $\langle Z e_k,e_k\rangle = 0$ it follows that $a_k=0$. Next,
$\langle Z e_{kl},e_{kl} \rangle = 0$ implies that $a_k+a_l + 2a_{kl} = 0$,
and hence $a_{kl}=0$ as well. Proofs of the remaining properties of the inner product are straightforward and hence omitted.

The proof of the proposition is then concluded by using the fact that all norms on the finite dimensional space
$\mathbb{R}_{\text{sym}}^{d\times d}$ are equivalent, with
equivalence constants depending on the dimension.
\end{proof}

We now prove that the Poisson bilinear form on GFEM and $a_W(\cdot,\cdot)$
are equivalent.
\begin{lemma}\label{lemma spectral equivalence}
The following equivalence relation holds
\[
\|\nabla u_G\|_{0,\Omega} \eqsim \|\varepsilon(\ot{u})\|_{0,\Omega}.
\]
\end{lemma}
\begin{proof}
Take $T\in \mathcal{T}_h$.
Note that $u_G$ is a quadratic function on $T$, namely
\[
u_G = \sum_{E\subset T} \Theta_E\varphi_E,\quad\mbox{where}\quad
\Theta_E=\Big(\varepsilon(\ot{u})\frac{x_E}{h},x_E \Big),
\]
This function vanishes at all
vertices of $T$, and its linear interpolant $(u_G)_I$ is zero. Thus,
\[
h^{-2}\|u_G\|_{0,T}^2 \lesssim \|\nabla u_G\|_{0,T}^2
\lesssim h^{-2} \|u_G\|_{0,T}^2,
\]
where the first inequality follows from the standard interpolation
error estimate and the second is an inverse inequality.
On the other hand
\[
\|u_G\|_{0,T}^2 \eqsim |T| \sum_{E\subset T} \Theta_E^2.
\]
Hence,
\begin{equation}\label{eqiv_relation}
|\nabla u_G|_T^2 \eqsim  h^{-2}|T|\sum_{E\subset T}
\Theta_E^2=|T|\sum_{E\subset T}\Big(\varepsilon(\ot{u})\frac{x_E}{h},\frac{x_E}{h}
\Big)^2.
\end{equation}
Note that the size of $\frac{x_E}{h}$ is independent of $h$ and thus
we can apply Proposition~\ref{proposition matrix} (with $\widetilde T$
homotetic to $T$, and with edges $y_E=\frac{x_E}{h}$). Hence, from
Proposition~\ref{proposition matrix} and \eqref{eqiv_relation}
we obtain
\begin{equation}
\begin{split}
\|\varepsilon(\ot{u})\|_{0,\Omega}^2 &=\sum_{T\in
\mathcal{T}_h}\|\varepsilon(\ot{u})\|_{0,T}^2=\sum_{T\in
\mathcal{T}_h}
\int_T\langle\varepsilon(\ot{u}),\varepsilon(\ot{u})\rangle_F\\
&\eqsim \sum_{T\in \mathcal{T}_h} |T|\sum_{E\subset T}
\Big(\varepsilon(\ot{u})\frac{x_E}{h},\frac{x_E}{h}\Big)^2
\eqsim \sum_{T\in\mathcal{T}_h} \|\nabla u_G\|_{0,T}^2\\
&= \|\nabla u_G\|_{0,\Omega}^2.
\end{split}
\end{equation}
\end{proof}

\section{Preconditioning for Linear Elasticity}

In this section, we develop and analyze an efficient preconditioner for the linear elasticity problem.
An additional spectral equivalence is needed to verify the auxiliary space lemma and, hence, the optimality for our preconditioner in this setting. This spectral equivalence result follows from a Korn inequality on
$\ot{H}\hskip0mm_\perp^1(\Omega)$ and the Cauchy-Schwarz inequality.
\begin{proposition}[Korn inequality] There exists a positive constant
$C$ such that
\[\|\varepsilon(\ot{v})\|_{0,\Omega} \ge
C\|\ot{v}\|_{1,\Omega}\quad \forall \ot{v}\in
\ot{H}\hskip0mm_\perp^1(\Omega).\]
\end{proposition}
\begin{proof} The proof is given in~\cite{Arnold_elasticity, Brenner_Scott_1994}.
\end{proof}
It follows from the Korn inequality that the weak formulation~\eqref{variational form} has a
unique solution $\ot{u}\in \ot{H}\hskip0mm_\perp^1(\Omega)$.
%%%%%%%%
\begin{theorem}\label{thm:spectralequiv_A LE} Further, the following spectral equivalence holds:
\[ \|\varepsilon(\ot{v})\|_{0,\Omega}^2 \lesssim  a_{LE}(\ot{v},\ot{v})\lesssim_\lambda \|\varepsilon(\ot{v})\|_{0,\Omega}^2,\quad
\forall \ot{v}\in \ot{H}\hskip0mm_\perp^1(\Omega).\]
\end{theorem}
\begin{proof} Let $\ot{v}\in \ot{H}\hskip0mm_\perp^1(\Omega)$. Then we
have
\begin{equation}
\begin{split}
a_{LE}(\ot{v},\ot{v})&=2\mu\|\varepsilon(\ot{v})\|_{0,\Omega}^2
+\lambda\|\nabla\cdot\ot{v}\|_{0,\Omega}^2\\
&\lesssim
2\mu\|\varepsilon(\ot{v})\|_{0,\Omega}^2
+\lambda\|\ot{v}\|_{1,\Omega}^2\\
&\lesssim 2\mu\|\varepsilon(\ot{v})\|_{0,\Omega}^2
+\lambda\|\varepsilon(\ot{v})\|_{0,\Omega}^2\\
&\lesssim_\lambda \|\varepsilon(\ot{v})\|_{0,\Omega}^2,
\end{split}\nonumber
\end{equation}
where the Cauchy-Schwarz inequality and the Korn inequality are
used. The reverse inequality is verified as follows:
\[\|\varepsilon(\ot{v})\|_{0,\Omega}^2\lesssim
2\mu\|\varepsilon(\ot{v})\|_{0,\Omega}^2 + \lambda
\|\nabla\cdot\ot{v}\|_{0,\Omega}^2=a_{LE}(\ot{v},\ot{v}),
\]
for $\forall \ot{v}\in
\ot{H}\hskip0mm_\perp^1(\Omega)$.
\end{proof}

Finally, we proceed to derive our GFEM auxiliary space preconditioner.  Consider
\[a_W(\ot{u},\ot{v}):=\int_\Omega
\langle\varepsilon(\ot{u}),\varepsilon(\ot{u})\rangle_F~dx
\quad\forall\ \ot{u}, \ot{v}\in W,\] and the mapping
$\Pi_G:W\rightarrow V^G$ defined as follows: let $\ot{u}\in
W$ be such that
\[\ot{u}=(u_k),\quad u_k=\sum_{i=1}^n \alpha_{k,i}\phi_i\quad\mbox{for}\quad k=1,
\cdots,d,\] then
\[ u_G=\Pi_G(\ot{u}):=\sum_{k=1}^2\sum_{i=1}^n\alpha_{k,i}\psi_{k,i}\phi_i \in V^G.\]
Let $V_q$ denote the space of continuous
piecewise quadratic finite elements equipped with standard inner product
$a_q(\cdot,\cdot)$ and define the mapping $\Pi_q : V^G \rightarrow
V_q$ by
\[ \Pi_q u_G := u_G, \quad\mbox{for}\quad u_G\in V^{GFEM}.\]
We note that $\Pi_q$ is the natural inclusion from $V^{GFEM}$ into
$V_q$. Associated with $a_q(\cdot,\cdot)$, we write
$A_q:V_q\rightarrow (V_q)'$ for its isomorphism. Going back to the
original problem~\eqref{variational form}, we define the conforming
finite element spaces on the mesh $\mathcal{T}_h$:
\[\mathcal{V}:=\{\ot{u}: \ot{u}=(u_k)_{k=1}^2, u_k\in
V_1\}\] and
\[\mathcal{V}_\perp:=\left\{\ot{u}\in \mathcal{V}: \int_\Omega \ot{u}\cdot\ot{v}~dx=0\quad
\forall \ot{v}\in \mathrm{RM}\right\}\subset
\ot{H}\hskip0mm_\perp^1(\Omega).\] We
remark that $W$ can be viewed as $W/\mathrm{RM}$
since $\mathrm{Ker}(\Pi_G)=\mathrm{RM}$ and then $W$ is isomorphic
to $\mathcal{V}_\perp$. To see this, define the operator $Q_{\mathrm{RM}}:W\rightarrow W$
as follows: for
$\ot{v}\in W$,
\[ (Q_{\tiny{\mathrm{RM}}}\ot{v},\ot{w})_{0,\Omega}
=(\ot{v},\ot{w})_{0,\Omega},\quad \ot{w}\in
\mathrm{RM} \subsetneqq W.\]
Using $Q_{\mathrm{RM}}$, we can regard $W$ as
$\mathrm{Range}(I-Q_{\mathrm{RM}})$ where $I$ is an
identity operator and $\mbox{Range}$ denotes
the range of an operator and, thus, $W$ is
isomorphic to $\mathcal{V}_\perp$.

The auxiliary space lemma
(See~\cite{Nepomnyaschikh_1992,JXu_aux} for a detailed proof) for the linear
elasticity problem reads:
\begin{lemma}[Auxiliary Space Lemma] Assume that $\widetilde{\Pi}:V_q\rightarrow
W$ is a surjective and bounded linear operator, namely, there exists
a positive constant $c_1$ such that $\forall v_q \in V_q$,
\[ a_W(\widetilde{\Pi}v_q,\widetilde{\Pi}v_q)\le c_1 a_q(v_q,v_q).\]
Also, we suppose that there exists a positive constant $c_0$ such
that $\forall \ot{v}\in W$, there is $v_q\in V_q$ so
that
\[ \ot{v}=\widetilde{\Pi}v_q\quad\mbox{and}\quad
a_q(v_q,v_q)\le c_0a_W(\ot{v},\ot{v})\] Then
\begin{eqnarray}\label{eq:specASL}
c_0^{-2}a_W(\ot{u},\ot{u})
\le a_W(\widetilde{\Pi}A_q^{-1}\widetilde{\Pi}^* A_W
\ot{u},\ot{u}) \le c_1^2
a_W(\ot{u},\ot{u}),\qquad \forall \ot{u}\in
W.
\end{eqnarray}
\end{lemma}
We recall that $a_W(\cdot,\cdot) \lesssim  a_{LE}(\ot{v},\ot{v})\lesssim_\lambda a_W(\cdot,\cdot)$.
Also, we see that $W$ and $\mathcal{V}_\perp$ are identical.
For these reasons, it suffices to construct a preconditioner for $a_W(\cdot,\cdot)$ on
the space $W$. Let $\Pi_q$ be an inclusion from $V^G$ into $V_q$. We define its Hilbert adjoint
operator $\Pi_q^\dag:V_q\rightarrow V^G$ by
\[
a_q(\Pi_q u_G, v_q)=a_G(u_G, \Pi_q^\dag v_q),\quad\forall u_G\in V^G, v_q\in V_q.
\]
We note that
\[
a_G(\Pi_G\ot{u},\Pi_G\ot{u})\doublebarsim a_W(\ot{u},\ot{u}),\quad \forall \ot{u}\in W,
\]
and
\[
a_G(v_G,v_G)\doublebarsim a_W(\Pi_G^{-1}v_G,\Pi_G^{-1}v_G),\quad v_G\in V^G,
\]
since $\Pi_G$ is bijective. Also, the Hilbert adjoint operator $(\Pi_G^{-1})^\dag$ of $\Pi_G^{-1}$ can be
defined as above. Associated with $a_{LE}(\cdot,\cdot)$
in~\eqref{variational form}, we write its isomorphism
$A_{LE}:\mathcal{V}_\perp\rightarrow (\mathcal{V}_\perp)'$.
We summarize the relations between these spaces and their associated interpolants and dual space in
the commutative diagram provided in Figure~\ref{fig:DeRahm}.

\begin{figure}\label{fig:DeRahm}
\[\begin{CD}
V_q @> A_q>> (V_q)' \\
@V\Pi_q^\dag VV  @AA(\Pi_q^\dag)^* A\\
V^G @>>A_G> (V^G)'\\
@V\Pi_G^{-1} VV  @AA(\Pi_G^{-1})^* A\\
W @>>A_W> W'\\
  @|      @|\\
\mathcal{V}_\perp @>>A_{LE}> (\mathcal{V}_\perp)'\\
\end{CD}\]
\caption{Commutative diagram for $A_{LE}$ and our auxiliary spaces.}
\end{figure}

Let $\widetilde{\Pi}:=\Pi_G^{-1}\Pi_q^\dag$. Then $\widetilde{\Pi}$ is surjective operator since $\Pi_q^\dag$ is surjective.
Letting $v_q\in V_q$, we have

\begin{eqnarray*}
a_W(\widetilde{\Pi} v_q,\widetilde{\Pi} v_q) &=& a_W(\Pi_G^{-1}\Pi_q^\dag v_q,\Pi_G^{-1}\Pi_q^\dag v_q)\\
&\lesssim& a_G(\Pi_q^\dag v_q,\Pi_q^\dag v_q)=a_q(\Pi_q\Pi_q^\dag v_q,v_q)\\
&\le& a_q(\Pi_q\Pi_q^\dag v_q,\Pi_q\Pi_q^\dag v_q)^{1/2}a_q(v_q,v_q)^{1/2}\\
&=& a_G(\Pi_q^\dag v_q,\Pi_q^\dag v_q)^{1/2}a_q(v_q,v_q)^{1/2}\\
&\lesssim&  a_W(\Pi_G^{-1}\Pi_q^\dag v_q,\Pi_G^{-1}\Pi_q^\dag v_q)^{1/2}a_q(v_q,v_q)^{1/2}\\
&=& a_W(\widetilde{\Pi} v_q,\widetilde{\Pi} v_q)^{1/2}a_q(v_q,v_q)^{1/2}.
\end{eqnarray*}
Further, let $\ot{v}\in W$ and set $v_q:=\Pi_q\Pi_G \ot{v} \in V_q$. Then, for $\forall \ot{u}\in W$,
we have
\begin{eqnarray*}
a_W(\widetilde{\Pi} v_q,\ot{u})&=& a_W(\Pi_G^{-1}\Pi_q^\dag\Pi_q\Pi_G\ot{v},\ot{u})\\
&=& a_G(\Pi_q^\dag\Pi_q\Pi_G\ot{v},(\Pi_G^{-1})^\dag\ot{u})\\
&=& a_q(\Pi_q\Pi_G\ot{v},\Pi_q(\Pi_G^{-1})^\dag\ot{u})\\
&=& a_G(\Pi_G\ot{v},(\Pi_G^{-1})^\dag\ot{u})\\
&=& a_G(\ot{v},\ot{u}).
\end{eqnarray*}
Therefore, $\ot{v}=\widetilde{\Pi} v_q$. Moreover,
\begin{eqnarray*}
a_q(v_q,v_q)&=&a_q(\Pi_q\Pi_G\ot{v},\Pi_q\Pi_G\ot{v})\\
&=& a_G(\Pi_G\ot{v},\Pi_G\ot{v})\\
&\lesssim& a_W(\ot{v},\ot{v}).
\end{eqnarray*}
Consequently, we obtain
\begin{equation}\label{condnumPCG}
\kappa(\Pi_G^{-1}\Pi_q^\dag A_q^{-1}(\Pi_q^\dag)^*(\Pi_G^{-1})^*A_W)
=\kappa(\widetilde{\Pi} A_q^{-1}\widetilde{\Pi}^*A_W)\lesssim 1.
\end{equation}
For notational brevity, let $B:=\widetilde{\Pi} A_q^{-1}\widetilde{\Pi}^*$. By \eqref{condnumPCG} and
Theorem~\ref{thm:spectralequiv_A LE}, we get $\kappa(BA_{LE})\lesssim \lambda$, independent of mesh size $h$.

We define the norm $\|\cdot\|_1$ on W as follows:
\[
\|\ot{u}\|_1^2:=|u_1|^2_1+|u_2|^2_1,
\]
where $\ot{u}=(u_1, u_2)\in W$ and write $A_1$ for the isomorphism associated with $\|\cdot\|_1^2$. Then, using the Korn inequality and the Cauchy-Schwarz inequality, we get
\begin{equation}\label{equivrel:A1andAW}
\int_\Omega \langle\varepsilon(\ot{u}),\varepsilon(\ot{u})\rangle_F~dx \doublebarsim \|\ot{u}\|_1^2,\qquad \forall \ot{u} \in W.
\end{equation}
These spectral equivalence relations we proposed basically motivate our choice of GFEM-based auxiliary space preconditioner. We also remark that the space $\ot{H}^1(\Omega)$ has the usual norm
$$
\|\ot{u}\|_{\ot{H}^1(\Omega)}:=\big(\|\ot{u}\|_1^2+
\|\ot{u}\|_{\ot{L}\hskip0mm_2}^2\big)^{1/2}.
$$

We describe an algorithm under which several numerical experiments in next subsection will be performed. Applying the shape functions on FEM into the bilinear forms $a_{LE}(\cdot,\cdot), a_W(\cdot,\cdot)$ and $a_1(\cdot,\cdot)$, we obtain
the Galerkin matrix $\boldsymbol{A}$, which represents matrices $\boldsymbol{A}_{LE}, \boldsymbol{A}_W$ and $\boldsymbol{A}_1$. We consider only a linear system
$$
\boldsymbol{A}x=f,
$$
where $f$ is appropriately chosen. This simplification makes sense by the spectral equivalent argument we established. The algorithm follows the preconditioned conjugate gradient methods in~\cite{Saad_book}.
\begin{algorithm}\label{algorithm:PCG}
\item Compute $r_0=f-\boldsymbol{A}x_0, z_0=\widetilde{\boldsymbol{\Pi}}\boldsymbol{B}_q
    {\widetilde{\boldsymbol{\Pi}}}^T r_0$, and $p_0=z_0$
\item For $j=0,1,...,$ until convergence Do:
\item \quad  $\alpha_j = (r_j,z_j)/(\boldsymbol{A}p_j,p_j)$
\item \quad  $x_{j+1}=x_j+\alpha_jp_j$
\item \quad  $r_{j+1}=r_j-\alpha_j\boldsymbol{A}p_j$
\item \quad  $z_{j+1}=\widetilde{\boldsymbol{\Pi}}\boldsymbol{B}_q
    {\widetilde{\boldsymbol{\Pi}}}^T r_{j+1}$
\item \quad  $\beta_j=(r_{j+1},z_{j+1})/(r_j,z_j)$
\item \quad  $p_{j+1}=z_{j+1}+\beta_j p_j$
\item EndDo
\end{algorithm}
\noindent Here, $\widetilde{\boldsymbol{\Pi}}$ is the matrix representation of $\widetilde{\Pi}$ and $\boldsymbol{B}_q$ is an approximate inverse of $\boldsymbol{A}_q$ where $\boldsymbol{A}_q$ is the Galerkin matrix arising from the piecewise quadratic FEM. The numerical experiments in next section are performed using the stopping criterion $\frac{\|r_k\|}{\|r_0\|}<10^{-8}$, where $r_k$ is the residual of $k$th iteration and initial guess $x_0=(1,-1,\ldots,1,-1)^T$. Taking an account into eigenvalues of the preconditioned system, we can estimate the condition number from parameters in the conjugate gradient algorithm (see~\cite{Saad_book} for more detail).

\section{Numerical experiments}

Here we report computed estimates of the condition number of our auxiliary-space preconditioned system matrix for the linear elasticity model discretized using various choices of mesh spacing on unit square $\Omega=[0,1]^2$. Here, $h$ represents the lengths of the horizontal and vertical sides of triangles in the meshes.
\begin{figure}[htp]
\begin{center}
\subfigure[$h=\frac{1}{2}$]{\includegraphics[width=1.2in]{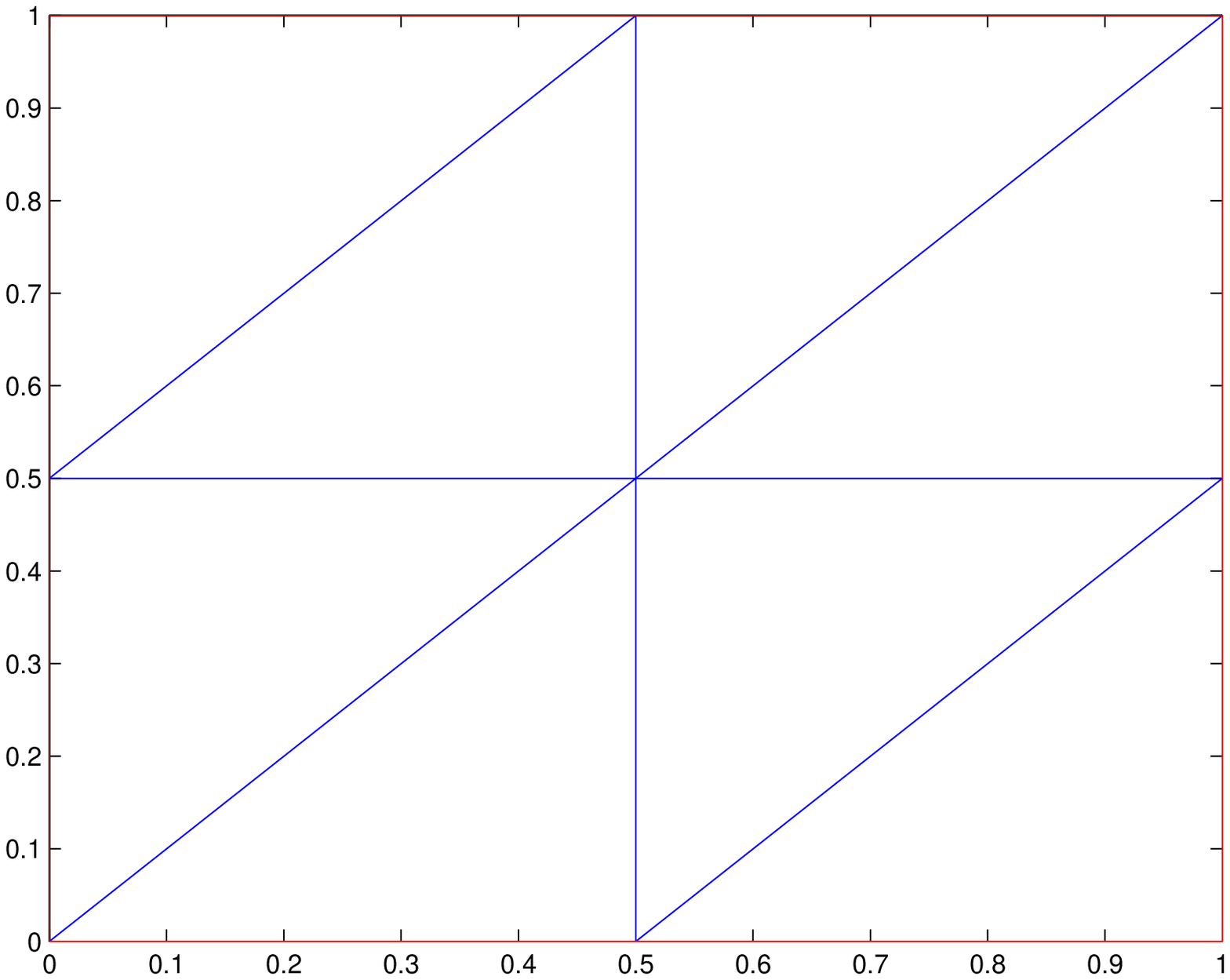}}
\subfigure[$h=\frac{1}{4}$]{\includegraphics[width=1.2in]{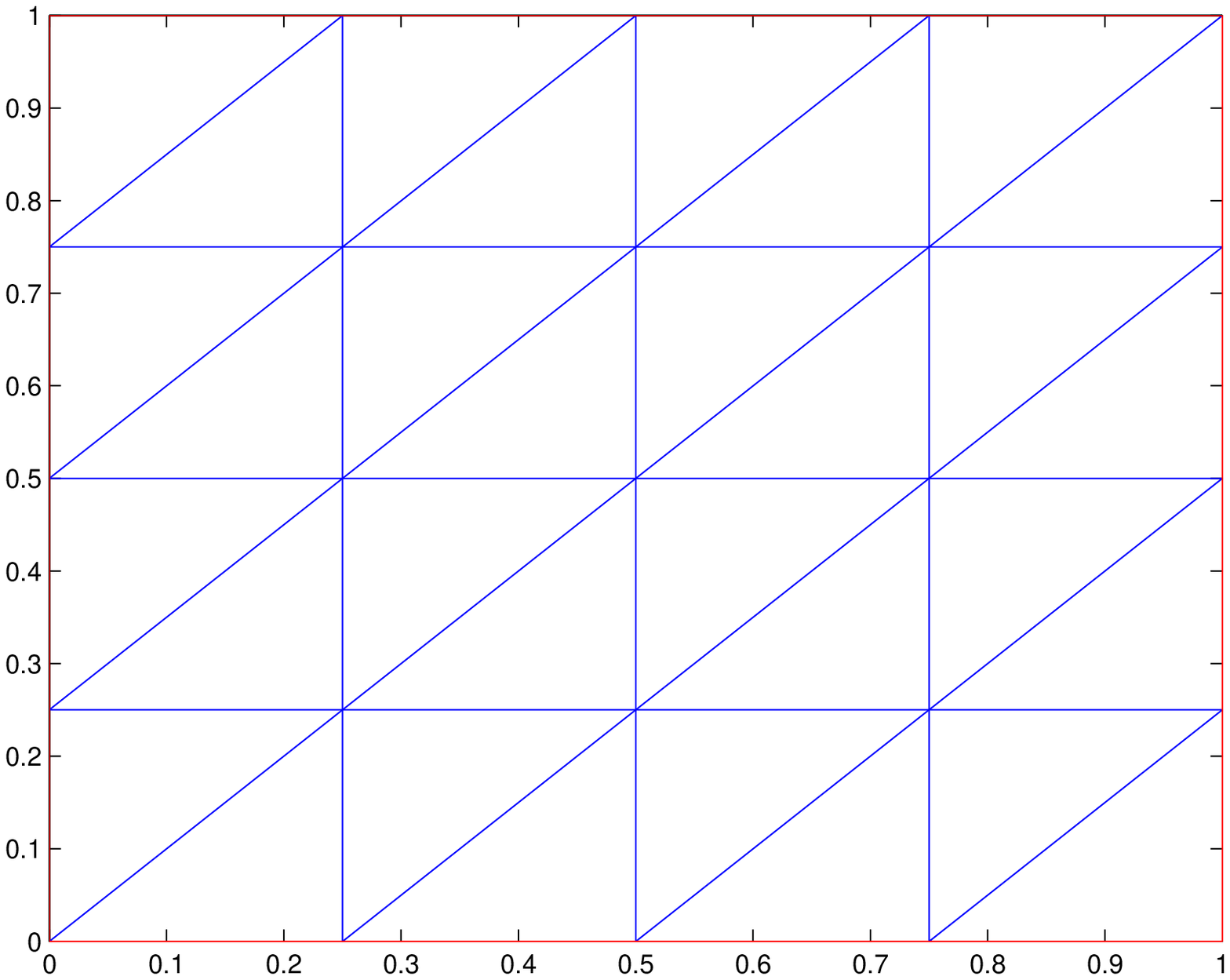}}
\subfigure[$h=\frac{1}{8}$]{\includegraphics[width=1.2in]{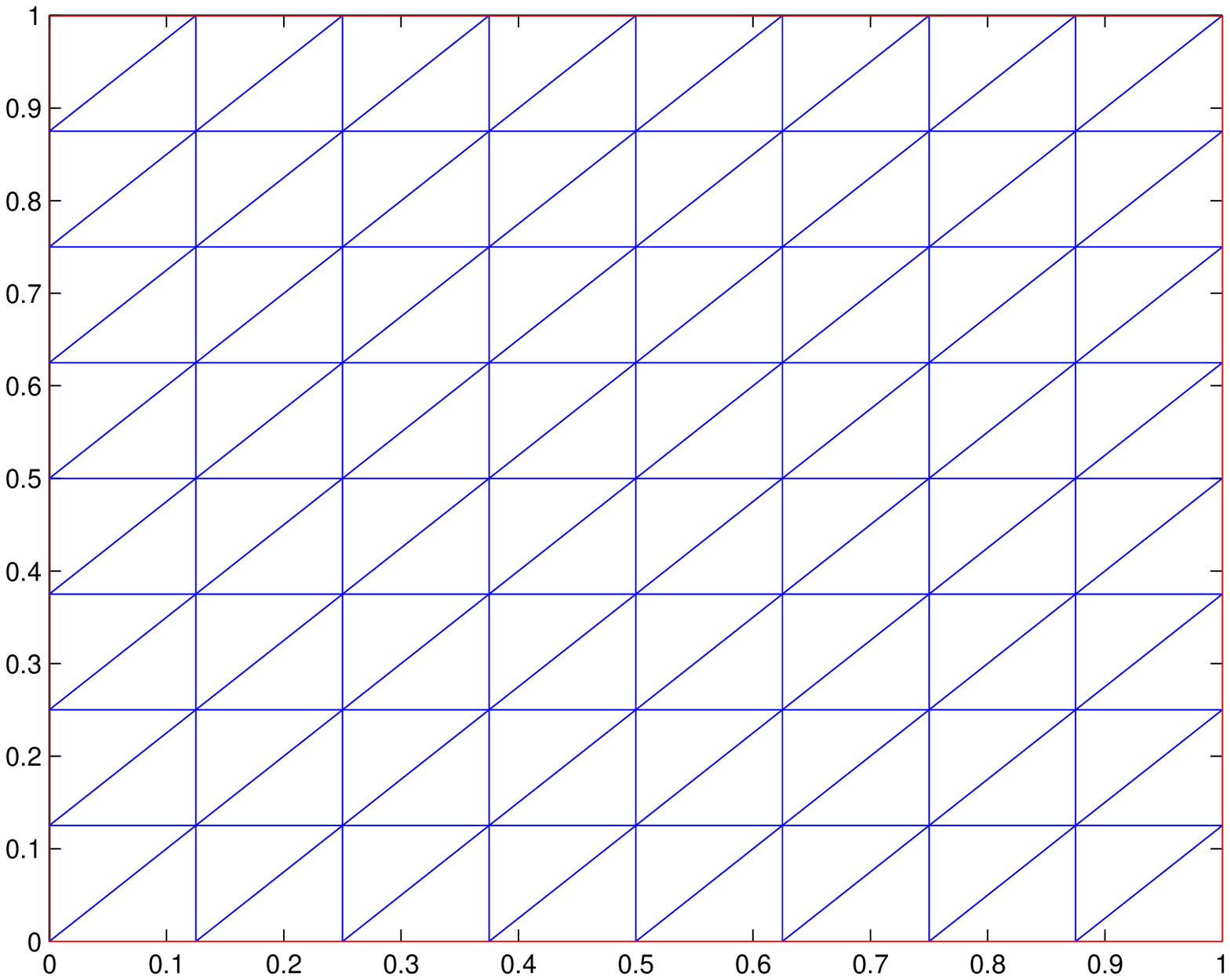}}
\subfigure[$h=\frac{1}{16}$]{\includegraphics[width=1.2in]{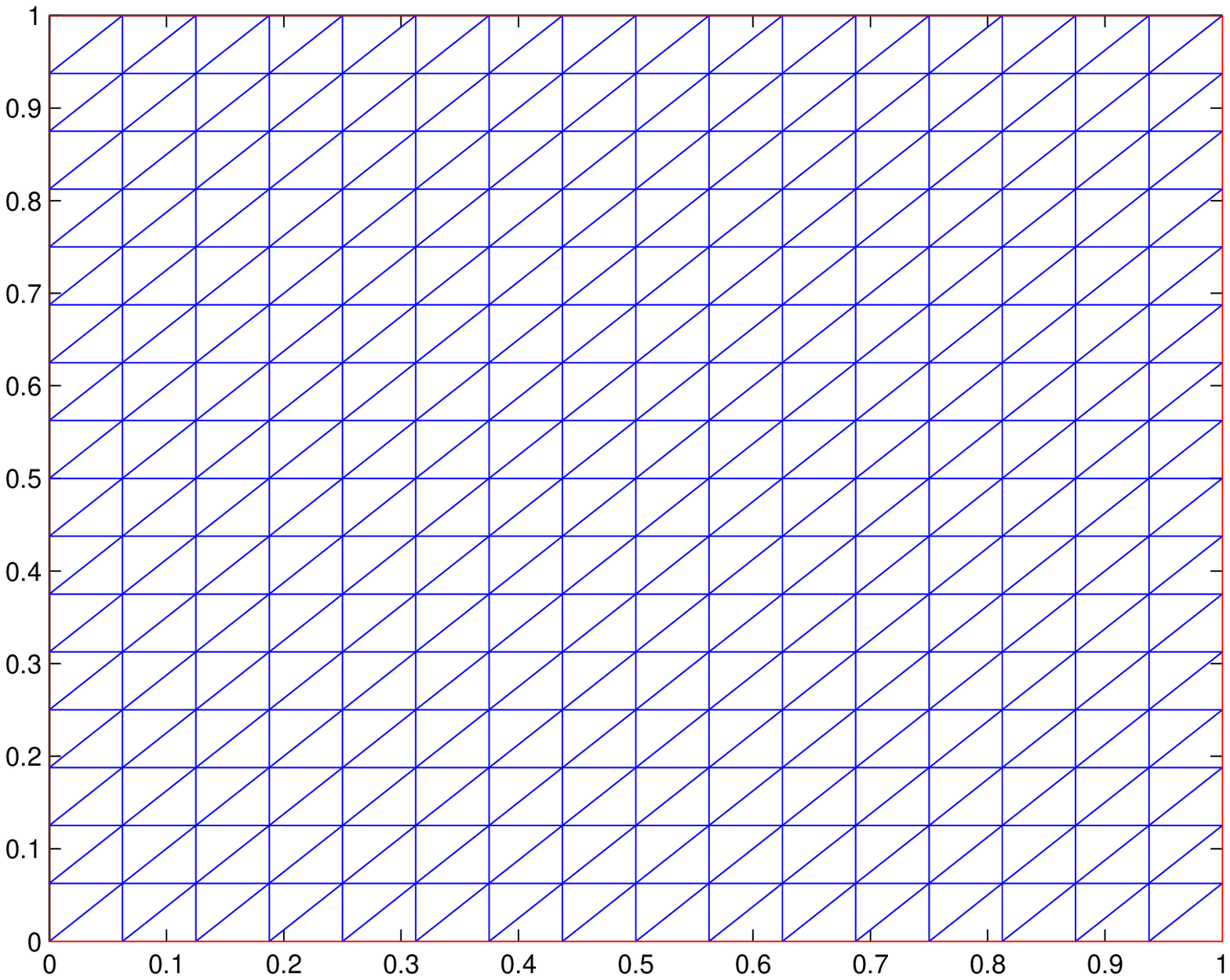}}
\subfigure[$h=\frac{1}{32}$]{\includegraphics[width=1.2in]{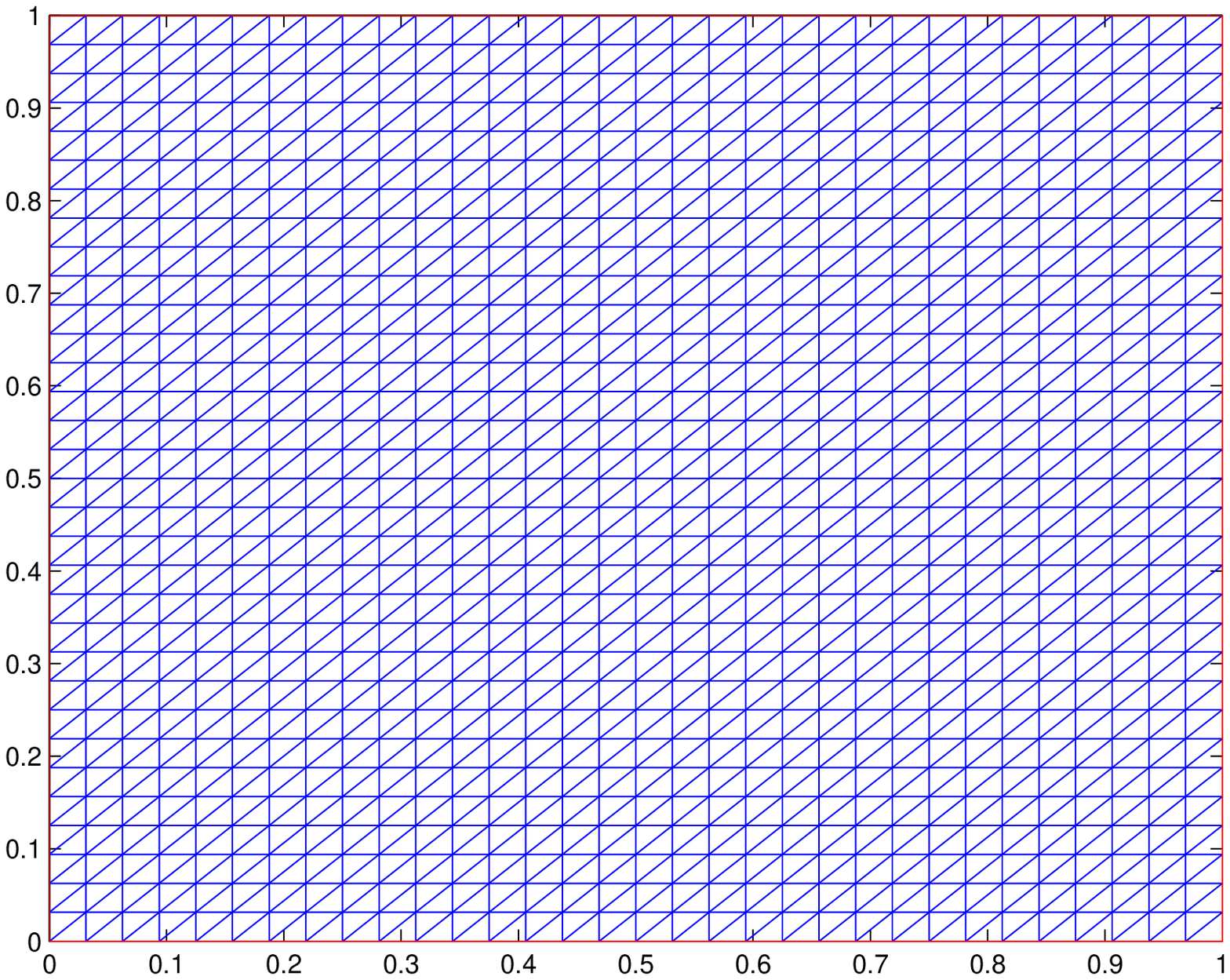}}
\end{center}
\caption{Meshes on the unit square}
\label{fig:meshes}
\end{figure}

{\bf Experiment 1.} We consider a pair of poisson equations on $\Omega$ with homogeneous
Neumann boundaries:
\begin{equation}\label{eqtn:A1}
\begin{cases}
\begin{split}
\begin{pmatrix}
\Delta & 0 \\
0 & \Delta
\end{pmatrix}
\ot{u}&=
\ot{0}\quad\mbox{in}\quad\Omega, \\
\begin{pmatrix}
\frac{\partial}{\partial n} & 0 \\
0 & \frac{\partial}{\partial n}
\end{pmatrix}
\ot{u} &=
\ot{0}\quad\mbox{on}\quad\partial\Omega.
\end{split}
\end{cases}
\end{equation}
We note that $A_1$ is the isomorphism associated with the weak formulation of~\eqref{eqtn:A1} on
the space $W$. The results in Table~\ref{table:condnum1} implies that the condition number of the preconditioned system of $A_1$ is bounded independent of mesh size $h$.
\begin{table}[h]
\begin{center}
\begin{tabular}{||c||c|c|c|c|c|c|c||}
  \hline\hline
  $h$&$\frac{1}{4}$&$\frac{1}{8}$&$\frac{1}{16}$&$\frac{1}{32}$
  &$\frac{1}{64}$\\
  \hline
  \mbox{iter}&20&24&27&27&27\\
  \hline
  $\kappa$&1.92e+1&2.80e+1&3.51e+1&4.02e+1&4.39e+1\\
  \hline\hline
\end{tabular}
\end{center}
\caption{Condition numbers $\kappa(\Pi_G^{-1}\Pi_q^\dag A_q^{-1}(\Pi_q^\dag)^*(\Pi_G^{-1})^*A_1)$.}\label{table:condnum1}
\end{table}

{\bf Experiment 2.} We deal with a special case when $\lambda=0$. Namely, the linear elasticity
reads
\begin{equation}\label{eqtn:A_W}
\begin{split}
-\ot{\rm{div}}\big(\varepsilon(\ot{u})\big)&=\ot{f}\quad \mbox{in}\ \Omega\\
\varepsilon(\ot{u})\ot{n}\Big|_{\Gamma_\ell}&=\ot{0},\quad 1\le
\ell\le 4,
\end{split}
\end{equation}
where $\Gamma_\ell$ represent four sides of the unit
square for $\ell=1,\ldots,4$. Also, without loss of generality, we suppose
that $\mu=\frac12$, here and later. Assume that
the body force $\ot{f}=\begin{pmatrix}f_1\\f_2
\end{pmatrix}$ is defined by
\begin{equation}
f_1=0\quad\mbox{and}\quad f_2=0.
\end{equation}
Then the exact solution $\ot{u}\in
\ot{H}\hskip0mm_\perp^2(\Omega)$ is
\[\ot{u}=\begin{pmatrix}0\\0\end{pmatrix}.\]
We see that $A_W$ is the isomorphism associated with the weak formulation of~\eqref{eqtn:A_W} on $W$.
In Table~\ref{table:condnum2}, $\kappa=C/c$, where $C$ and $c$ are the
smallest and largest constants satisfying:
\[
c a_W(\ot{u},\ot{u}) \le (B_q\Pi_q\Pi_G
\ot{u},\Pi_q\Pi_G \ot{u})_0 \le C
a_W(\ot{u},\ot{u}),
\forall \ot{u}\in \mathcal{V}_\perp.\\
\]
\begin{table}[h]
\begin{center}
\begin{tabular}{||c||c|c|c|c|c|c|c||}
  \hline\hline
  $h$&$\frac{1}{4}$&$\frac{1}{8}$&$\frac{1}{16}$&$\frac{1}{32}$
  &$\frac{1}{64}$ \\
  \hline
  \mbox{iter}&19&19&17&16&15 \\
  \hline
  $\kappa$&6.47e+0&6.70e+0&6.71e+0&6.67e+0&6.48e+0\\
  \hline\hline
\end{tabular}
\end{center}
\caption{Condition numbers $\kappa(\Pi_G^{-1}\Pi_q^\dag A_q^{-1}(\Pi_q^\dag)^*(\Pi_G^{-1})^*A_W)$.}\label{table:condnum2}
\end{table}
In Table~\ref{table:condnum2}, we observe that the preconditioner we devised is an efficient
preconditioner for the special case. Namely, $\kappa(BA_W)\lesssim 1$.

{\bf Experiment 3.} In this experiment, we study more general linear elasticity problems
with various $\lambda$:
\begin{equation}
\begin{split}
-\ot{\rm{div}}\big(\sigma(\ot{u})\big)&=\ot{0}\quad \mbox{in}\ \Omega=\mbox{unit\ square}\\
\sigma(\ot{u})\ot{n}\Big|_{\Gamma_\ell}&=\ot{0}\quad 1\le
\ell\le 4,
\end{split}
\end{equation}
where $\sigma(\ot{u})=\varepsilon(\ot{u})
+\lambda\mathrm{tr}(\varepsilon(\ot{u}))\ott{\delta}$
and $\Gamma_\ell$ are defined as in Experiment 2. In Tables \ref{table:condnum3}, \ref{table:condnum4}
and \ref{table:challenge}, $\kappa=C/c$, where $C$ and $c$ are the
smallest and largest constants satisfying:
\[
c a_{LE}(\ot{u},\ot{u}) \le (B_q\Pi_q\Pi_G
\ot{u},\Pi_q\Pi_G \ot{u})_0 \le C
a_{LE}(\ot{u},\ot{u}),
\forall \ot{u}\in \mathcal{V}_\perp.
\]
\begin{table}[h]
\begin{center}
\begin{tabular}{|c||c|c|c|c|c|c|}
  \hline
  & \multicolumn{2}{c|}{$\lambda=1$}&\multicolumn{2}{c|}{$\lambda=5$}
  &\multicolumn{2}{c|}{$\lambda=10$}\\
  \cline{2-7}
  $h$& iter & $\kappa$ & iter & $\kappa$ & iter & $\kappa$ \\
  \hline\hline
  $\frac{1}{4}$&24&1.33e+1&30&4.60e+1&33&8.60e+1  \\
  \hline
  $\frac{1}{8}$&28&1.39e+1&39&4.51e+1&50&8.60e+1 \\
  \hline
  $\frac{1}{16}$&25&1.40e+1&43&4.57e+1&55&8.55e+1 \\
  \hline
  $\frac{1}{32}$&23&1.40e+1&40&4.56e+1&53&8.55e+1 \\
  \hline
  $\frac{1}{64}$&21&1.39e+1&36&4.54e+1&46&8.49e+1 \\
  \hline
\end{tabular}
\end{center}
\caption{Condition numbers $\kappa(\Pi_G^{-1}\Pi_q^\dag A_q^{-1}(\Pi_q^\dag)^*(\Pi_G^{-1})^*A_{LE})$.}\label{table:condnum3}
\end{table}%GeneralLinearElasticity.m

{\bf Experiment 4.} For small $\lambda$, we experiment the linear elasticity with homogeneous
pure traction boundary condition:
\begin{equation}\label{eqtn:generalsolution}
\begin{split}
-\ot{\rm{div}}\big(\varepsilon(\ot{u})
+\lambda\mathrm{tr}(\varepsilon(\ot{u}))\ott{\delta}\big)&=\ot{f}\quad \mbox{in}\ \Omega\\
\big(\varepsilon(\ot{u})
+\lambda\mathrm{tr}(\varepsilon(\ot{u}))\ott{\delta}\big)\ot{n}\Big|_{\Gamma_\ell}&=\ot{0}\hskip0mm_\ell\quad
1\le \ell\le 4,
\end{split}
\end{equation}
where $\Gamma_\ell$ are also defined as in Experiment 2 and the body force $\ot{f}=\begin{pmatrix}f_1\\f_2
\end{pmatrix}$ is such that
$$
\ot{u}=\begin{pmatrix}u_1\\u_2\end{pmatrix}=
\begin{pmatrix}x(1-x)y^2(1-y)^2\sin{\pi x}-\frac{2}{15\pi^3}\\  x^2(1-x)^2y^2(1-y)^2\cos{\pi y}\end{pmatrix}\in
\ot{H}\hskip0mm_\perp^2(\Omega)
$$
is the exact solution of~\eqref{eqtn:generalsolution}. The results reported in Tables~\ref{table:condnum4} indicate that the condition number of our preconditioned operator remains bounded independent of mesh size and, hence, provides for an optimal order preconditioner for our linear elasticity model problem. They coincide with our theoretical result that $\kappa(BA_{LE}) \lesssim \lambda$, independent of $h$. We note that although our experiments have been carried out on a structured grid, our analysis carries over to unstructured simplicial meshes and our scheme is thus expected to perform similar for unstructured grids as well.
\begin{table}[h]
\begin{center}
\begin{tabular}{|c||c|c|c|c|c|c|c|c|c|}
  \hline
  & \multicolumn{3}{c|}{$\lambda=1$}&\multicolumn{3}{c|}{$\lambda=5$}
  &\multicolumn{3}{c|}{$\lambda=10$} \\
  \cline{2-10}
  $h$& iter & $\kappa$ &$\ot{H}^1$ norm & iter &$\kappa$&$\ot{H}^1$ norm & iter & $\kappa$
  &$\ot{H}^1$ norm \\
  \hline\hline
  $\frac{1}{4}$&26&1.42e+1&7.71e-3&32&4.60e+1&1.10e-2&35&8.60e+1&1.25e-2\\
  \hline
  $\frac{1}{8}$&31&1.39e+1&4.07e-3&46&4.60e+1&7.56e-3&56&8.60e+1&9.62e-3\\
  \hline
  $\frac{1}{16}$&31&1.41e+1&1.41e-3&53&4.58e+1&3.48e-3&68&8.56e+1&5.21e-3\\
  \hline
  $\frac{1}{32}$&31&1.41e+1&3.91e-4&54&4.59e+1&1.12e-3&73&8.58e+1&1.92e-3\\
  \hline
  $\frac{1}{64}$&31&1.41e+1&1.01e-4&54&4.59e+1&3.04e-4&72&8.58e+1&5.49e-4\\
  \hline
\end{tabular}
\end{center}
\caption{Condition numbers $\kappa(\Pi_G^{-1}\Pi_q^\dag A_q^{-1}(\Pi_q^\dag)^*(\Pi_G^{-1})^*A_{LE})$ and the discrete $\ot{H}^1$ errors between the exact solution and the numerical solution.}\label{table:condnum4}
\end{table}%General_LE_TractionBD.m

{\bf Challenging experiment.} It is well known that the performance of the
piecewise linear finite elements deteriorates as $\lambda$ approaches $\infty$. In the elasticity literature, it is called the locking phenomenon (see~\cite{Babuska_Suri1992a,Babuska_Suri1992b} for more information). To overcome the effects of the locking, several methods have been suggested in~\cite{Arnold_Brezzi_Falk_Marini,Brenner_1994,Cai_Manteuffel_McCormick1998,Cai_Lee_Manteuffel_McCormick2000,
Falk_1991,Lee_1998_2,Schoberl}. Both our theoretical and numerical results show that the GFEM-based auxiliary preconditioner for the linear elasticity problem with small $\lambda$ works very efficiently. It would also be worthwhile to investigate the numerical results when the preconditioner is applied to linear elasticity for large $\lambda$. We observe in Table~\ref{table:challenge} that
even when $\lambda$ is large enough, the condition number of the preconditioned systems is bounded independent of mesh size $h$ and the discretization error estimates look reasonable if $h$ is sufficiently small.
\begin{table}[h]
\begin{center}
{\tiny
\begin{tabular}{|c||c|c|c|c|c|c|c|c|c|c|c|c|}
  \hline
  &\multicolumn{3}{c|}{$\lambda=50~(\nu=0.49505)$}
  &\multicolumn{3}{c|}{$\lambda=100~(\nu=0.49751)$}
  &\multicolumn{3}{c|}{$\lambda=500~(\nu=0.49950)$}
  &\multicolumn{3}{c|}{$\lambda=1000~(\nu=0.49975)$}\\
  \cline{2-13}
  $h$& iter & $\kappa$ &$\ot{H}^1$ norm &
  iter & $\kappa$ &$\ot{H}^1$ norm & iter & $\kappa$ &$\ot{H}^1$ norm & iter &$\kappa$&$\ot{H}^1$ norm\\
  \hline\hline
  $\frac{1}{8}$&92&4.06e+2&1.39e-2&109&8.06e+2&1.51e-2&134&4.01e+3&1.65e-2&140&8.01e+3&1.66e-2\\
  \hline
  $\frac{1}{16}$&114&4.06e+2&1.04e-2&144&8.06e+2&1.25e-2&224&4.01e+3&1.57e-2&258&8.01e+3&1.64e-2\\
  \hline
  $\frac{1}{32}$&135&4.06e+2&5.75e-3&170&8.06e+2&8.06e-3&287&4.01e+3&1.31e-2&365&8.01e+3&1.47e-2\\
  \hline
  $\frac{1}{64}$&145&4.05e+2&2.19e-3&183&8.06e+2&3.69e-3&322&4.01e+3&8.78e-3&404&8.01e+3&1.12e-2\\
  \hline
  $\frac{1}{128}$&143&4.05e+2&6.42e-4&192&8.05e+2&1.21e-3&367&4.00e+3&4.32e-3&455&8.01e+3&6.36e-3\\
  \hline
\end{tabular}}
\end{center}
\caption{Condition numbers $\kappa(\Pi_G^{-1}\Pi_q^\dag A_q^{-1}(\Pi_q^\dag)^*(\Pi_G^{-1})^*A_{LE})$ and the discrete $\ot{H}^1$ errors between the exact solution and the numerical solution.}\label{table:challenge}
\end{table}%General_LE_TractionBD.m

\section{Concluding remarks}

In this paper, using the auxiliary space technique we have designed a preconditioner for the solution of the problem of linear elasticity, which is also based on generalized finite element methods. We have proved that for arbitrarily fixed $\lambda$, the GFEM-based auxiliary preconditioner always works optimally for the system of linear elasticity discretized using the lowest finite element shape functions.

\bibliographystyle{plain}
\bibliography{GFEM0}
\end{document}